\documentclass[a4paper,12pt,final]{amsart}
\usepackage{times,a4wide,mathrsfs,showkeys,amssymb,amsmath,amsthm,enumerate,xypic,tikzsymbols,dsfont}

\newcommand{\C}{\mathbb{C}}
\newcommand{\ZZ}{\mathcal{Z}}

\newcommand{\QQ}{\mathbb{Q}}
\newcommand{\NN}{\mathbb{N}}
\newcommand{\PP}{\mathbb{P}}

\newcommand{\GG}{\mathcal G}

\newcommand{\YY}{\mathcal Y}

\newcommand{\CC}{\mathcal C}

\newcommand{\MM}{\mathcal M}

\newcommand{\FF}{\mathcal F}

\newcommand{\wt}{\widetilde}

\newcommand{\one}{\mathds{1}}

\DeclareMathOperator{\ide}{id}

\DeclareMathOperator{\ima}{Im}

\DeclareMathOperator{\Gr}{Gr}

\newtheorem{theorem}{Theorem}[section]

\newtheorem{lemma}[theorem]{Lemma}

\newtheorem{corollary}[theorem]{Corollary}
\newtheorem{proposition}[theorem]{Proposition}
\newtheorem{conjecture}[theorem]{Conjecture}
\newtheorem{remark}[theorem]{Remark}

\newtheorem{convention}{Conventions}

\newtheorem{notation}[theorem]{Notation}

\newtheorem{nonumbering}{Theorem}

\newtheorem{nonumberingc}{Corollary}

\newtheorem{nonumberingt}{Acknowledgements}

\newtheorem{nonumberingcon}{Statements and declarations}

\begin{document}

\author[Michele Bolognesi]{Michele Bolognesi}
\address{Institut Montpellierain Alexander Grothendieck \\ %
Universit\'e de Montpellier \\ %
CNRS \\ %
Case Courrier 051 - Place Eug\`ene Bataillon \\ %
34095 Montpellier Cedex 5 \\ %
France}
\email{michele.bolognesi@umontpellier.fr}

\author[Robert Laterveer]
{Robert Laterveer}

\address{Institut de Recherche Math\'ematique Avanc\'ee,
CNRS -- Universit\'e 
de Strasbourg,\
7 Rue Ren\'e Des\-car\-tes, 67084 Strasbourg CEDEX,
FRANCE.}
\email{robert.laterveer@math.unistra.fr}

\title{A 9-dimensional family of K3 surfaces with finite-dimensional motive}

\begin{abstract} Let $S$ be a K3 surface obtained as triple cover of a quadric branched along a genus 4 curve. Using the relation with cubic fourfolds, we show that $S$
 has finite-dimensional motive, in the sense of Kimura. We also establish the Kuga--Satake Hodge conjecture for $S$, as well as Voisin's conjecture concerning zero-cycles.
 As a consequence, 
we obtain Kimura finite-dimensionality, the Kuga--Satake Hodge conjecture, and Voisin's conjecture for 2 (9-dimensional) irreducible components of the moduli space of K3 surfaces with an order 3 non-symplectic automorphism.
 \end{abstract}

\thanks{MB and RL are supported by ANR grant ANR-20-CE40-0023.}


\maketitle

\section{Introduction}

The interplay between cubic fourfolds and K3 surfaces has since long been an important source of new, beautiful geometric constructions and results. Many crucial conjectures (e.g. \cite{kuz}) turn around the relations between these objects, and most interesting invariants of the one are related to those of the other (\cite{has}, \cite{BP}). However, heuristically speaking, the stream of information has almost always been flowing from the \it K3 side \rm to the \it cubic fourfold side. \rm This is quite natural, since somehow K3 surfaces have been an important object of research for longer than cubic fourfolds, and are in many respects better understood.

In this paper we reverse a bit the perspective: by using some recent results about the motives of cyclic cubic fourfolds we obtain a proof of the finite-dimensionality of the motive of certain K3 surfaces. Indeed, the goal of this paper is to show the following theorem.

\begin{nonumbering}[=Theorem \ref{main}] Let $S\subset\PP^4$ be a K3 surface defined as a smooth complete intersection with equations
  \[ \begin{cases}  &f(x_0,\ldots,x_3)=0\\
                            &g(x_0,\ldots,x_3)+ x_4^3=0\ ,\\
                            \end{cases}\]
                       where $f$ and $g$ are homogeneous polynomials of degree 2 resp. 3. Then $S$ has finite-dimensional motive, in the sense of Kimura \cite{Kim}.
                 \end{nonumbering}

To prove this, we use in a substantial way a nice result of Boissière-Heckel-Sarti, who show that certain cuspidal cyclic cubic fourfolds have a Fano variety of lines $F$ that is singular exactly along a $(2,3)$-complete intersection $S$ as in Theorem \ref{main}. Moreover, the blow-up of $F$ along $S$ is smooth and isomorphic to the Hilbert scheme $S^{[2]}$. This allows us to use some motivic technical results and extend the finite-dimensionality of the motive of cyclic cubic fourfolds (known thanks to the work of the second named author \cite{cubic}) to the motive of our family of degree 6 K3 surfaces.

In a similar vein, we establish the Kuga--Satake Hodge conjecture for these K3 surfaces:

\begin{nonumbering}[=Theorem \ref{main2}] Let $S$ be a K3 surface as in Theorem \ref{main}, and let $A$ be its associated Kuga--Satake variety. The injection
  \[ H^2(S,\QQ)\ \hookrightarrow\ H^2(A\times A,\QQ)\]
  is induced by a cycle on $S\times A\times A$.
  \end{nonumbering}
  
This relies on the relation with cyclic cubic fourfolds, combined with the fact that the Kuga--Satake Hodge conjecture is known for cyclic cubic fourfolds \cite{vGI}.

Combining Theorems \ref{main} and \ref{main2}, we obtain the verification of an old conjecture of Voisin \cite{Vo0} for these K3 surfaces:

\begin{nonumbering}[=Theorem \ref{main3}] Let $S$ be a K3 surface as in Theorem \ref{main}. For any two degree zero 0-cycles $a,b\in A^2_{hom}(S)$, there is equality
  \[  a\times b = b\times a\ \ \ \hbox{in}\ A^4(S\times S)\ .\]
  \end{nonumbering}

\medskip
   
Surfaces $S$ as in Theorem \ref{main} form a 9-dimensional family.
Thanks to work of Artebani--Sarti \cite{AS}, this family constitutes a Zariski open subset of an irreducible component of the moduli space of K3 surfaces with an order 3 non-symplectic automorphism. Using this, we can deduce the following:               
                            
\begin{nonumberingc}[cf. Section \ref{final}] Let $S$ be a K3 surface with an order 3 non-symplectic automorphism $\sigma$. Assume that the fixed locus of $\sigma$ contains a curve. Then $S$ has finite-dimensional motive, in the sense of Kimura \cite{Kim}. Moreover, the Kuga--Satake Hodge conjecture and Voisin's conjecture are true for $S$.
\end{nonumberingc}

(Unfortunately, we have not been able to get rid of the assumption on the fixed locus; cf. Remark \ref{pity} below.) We observe that this corollary is not a rephrasing of the preceding result, since the sextic K3 surfaces described so far describe only (the generic) elements of one component of the moduli space. We show that the Corollary holds also for a second component, whose general element is the Jacobian elliptic fibration of an element of the first component.

\smallskip

Finally, let us mention the following remarkable consequence of Kimura finite-dimensionality:

\begin{nonumberingc} Let $S$ be as in Corollary \ref{cor}, and let $X=S^{[m]}$ be the Hilbert scheme of length $m$ 0-dimensional subschemes of $S$. Then the Beauville--Voisin conjecture is true for $X$, i.e. the $\QQ$-subalgebra
  \[ R^\ast(X):=\bigl\langle A^1(X), c_j(X)\bigr\rangle\ \ \subset\ A^\ast(X) \]
  generated by divisors and Chern classes of the tangent bundle of $X$ injects into cohomology.
  \end{nonumberingc}
  
  This follows from work of Yin \cite{Yin}, cf. Theorem \ref{yin} below.
   
 \vskip0.6cm

\begin{convention} In this article, the word {\sl variety\/} will refer to a reduced irreducible scheme of finite type over $\C$. A {\sl subvariety\/} is a (possibly reducible) reduced subscheme which is equidimensional. 

{\bf All Chow groups will be with rational coefficients}: we will denote by $A_j(Y)$ the Chow group of $j$-dimensional cycles on $Y$ with $\QQ$-coefficients; for $Y$ smooth of dimension $n$ the notations $A_j(Y)$ and $A^{n-j}(Y)$ are used interchangeably. 
The notation $A^j_{hom}(Y)$ will be used to indicate the subgroup of homologically trivial cycles.
For a morphism $f\colon X\to Y$, we will write $\Gamma_f\in A_\ast(X\times Y)$ for the graph of $f$.

The contravariant category of Chow motives (i.e., pure motives with respect to rational equivalence as in \cite{Sc}, \cite{MNP}) will be denoted 
$\MM_{\rm rat}$.
\end{convention}

\section{Preliminaries}

\subsection{Finite-dimensionality}

There are plenty of good references \cite{Kim}, \cite{J4},  \cite{MNP} for the definition of finite–dimensional motive.
A crucial property of varieties that have finite-dimensional motive is certainly the nilpotence theorem.

\begin{theorem}(Kimura \cite{Kim})
Let $X$ be a smooth projective variety of dimension $n$ with finite dimensional motive. Let $\Gamma \in A^n(X\times X)$ be a numerically trivial correspondence. Then there exists $m\in \mathbb{N}$ such that

$$\Gamma^{\circ m}=0 \ \ \in \ A^n(X\times X).$$

\end{theorem}

More precisely, the nilpotence property for all the powers of $X$ could even be used as an alternative definition of finite-dimensionality for a motive (see e.g. Jannsen \cite[Corollary 3.9]{J4}). Kimura has conjectured that all projective variety has finite dimensional motive \cite{Kim}, but this is of course far from being proved. Nevertheless there exists a bunch of interesting examples.

\begin{remark}\label{examples}
The following varieties are known to have a finite-dimensional motive:

\begin{enumerate}

\item varieties that are dominated by products of curves \cite{Kim}, and varieties of dimension $\leq 3$ rationally dominated by products of curves \cite[Example 3.15]{V3} (in particular, the K3 surfaces studied in \cite{Par0} and the K3 surfaces in \cite{ILP});

\item K3 surfaces with Picard number 19 or 20 \cite{Ped};

\item K3 surfaces obtained as complete intersections of 3 diagonal quadrics in $\PP^5$ \cite{18};
 
\item surfaces not of general type with vanishing geometric genus \cite[Theorem 2.11]{GP};

\item many examples of surfaces of general type with $p_g=0$ \cite{PW};

\item Hilbert schemes of surfaces known to have finite-dimensional motive \cite{dCM}; 

\item Fano varieties of lines of smooth cubic threefolds, and Fano varieties of lines of smooth cubic fivefolds \cite{22}; 

\item generalized Kummer varieties \cite[Remark 2.9]{Zu};

 \item 3-folds with nef tangent bundle \cite[Example 3.16]{V3}, and certain 3-folds of general type \cite[Section 8]{Vial34fold}; 

\item varieties $X$ with Abel-Jacobi trivial Chow groups (i.e. $A^k_{AJ}(X)=0$ for all $k$) \cite[Theorem 4]{43};

\item products of varieties with finite-dimensional motive \cite{Kim}.

\end{enumerate} 
 
\end{remark}

\begin{remark}
It is worth pointing out that so far, all the examples of finite-dimensional Chow motives happen to be of abelian type. This is the tensor subcategory generated by Chow motives of curves. On the other hand, many very important examples do not lie inside this subcategory, e.g. the motive of a general hypersurface in $\PP^3$ \cite[Section 7]{Del}.
\end{remark}

For K3 surfaces, finite-dimensionality has a remarkable consequence:

\begin{theorem}[Yin \cite{Yin}]\label{yin} Let $S$ be a K3 surface, and assume $S$ has finite-dimensional motive. Then the Beauville--Voisin conjecture is true for the Hilbert schemes $X:=S^{[m]}$ for all $m\in\NN$, i.e. the $\QQ$-subalgebra
  \[ R^\ast(X):=\bigl\langle A^1(X), c_j(X)\bigr\rangle\ \ \subset\ A^\ast(X) \]
  generated by divisors and Chern classes of the tangent bundle of $X$ injects into cohomology.
\end{theorem}

\begin{proof} This is \cite[Corollary]{Yin}.
\end{proof}

\subsection{K3 surfaces with an order 3 non-symplectic automorphism}

\begin{notation} As in \cite[Section 5]{AS}, let $\MM_{n,k}$ denote the moduli space of K3 surfaces $S$ with an order 3 non-symplectic automorphism $\sigma$ such that the fixed locus of $\sigma$ consists of $n$ points and $k$ irreducible curves.
\end{notation}

\begin{theorem}(Artebani--Sarti \cite{AS})\label{as} The moduli space of K3 surfaces with an order 3 non-symplectic automorphism has 3 irreducible components (of dimension 9, 9 and 6), which are the closures of
  \[ \MM_{0,1}\ ,\ \MM_{0,2}\ ,\ \MM_{3,0}\ .\]
\end{theorem}

\begin{proof} This is \cite[Theorem 5.6]{AS}.
\end{proof}

One has explicit descriptions of the various $\MM_{n,k}$; in particular, $\MM_{0,1}$ consists of the surfaces that are of interest to us:

\begin{proposition}(Artebani--Sarti \cite{AS}) Any surface $S\in\MM_{0,1}$ is isomorphic to a complete intersection in $\PP^4$ of the form
  \[ \begin{cases}  &f(x_0,\ldots,x_3)=0\\
                            &g(x_0,\ldots,x_3)+ x_4^3=0\ ,\\
                            \end{cases}\]
                       where $f$ and $g$ are homogeneous polynomials of degree 2 resp. 3.
                       
    Conversely, the generic complete intersection of this form is in $\MM_{0,1}$.                   
\end{proposition}

\begin{proof} This is part of \cite[Proposition 4.7]{AS}.
\end{proof}

\begin{remark} The K3 surfaces in $\MM_{0,1}$ admit an alternative decription as triple cover of a quadric branched along a genus 4 curve \cite[Remark 4.8]{AS}. Conversely, any triple cover of a quadric branched along a (smooth) genus 4 curve is in $\MM_{0,1}$ \cite[Theorem 1]{Kon}. Very recently, a new description of the component $\MM_{3,0}$ has been given in \cite{MT}.
\end{remark}

There is a nice relation between the moduli spaces $\MM_{0,1}$ and $\MM_{0,2}$:

\begin{proposition}(Kondo \cite{Kon})\label{kon} Any surface in $\MM_{0,1}$ admits an elliptic fibration. The general surface $S\in\MM_{0,2}$ is the Jacobian elliptic fibration associated to some $S^\prime\in\MM_{0,1}$.
\end{proposition}

\begin{proof} This is contained in \cite[Section 4]{Kon}.
\end{proof}

\subsection{Cuspidal cyclic cubic fourfolds and K3 surfaces}

\begin{theorem}[Boissi\`ere--Heckel--Sarti \cite{BHS}]\label{bhs} Let $Y\subset\PP^5$ be a cuspidal cyclic cubic fourfold, i.e. a cubic defined by an equation
  \[    x_0 f(x_1,\ldots,x_4) + g(x_1,\ldots,x_4) + x_5^3=0\ ,\]
  where $f$ and $g$ are homogeneous of degree 2 resp. 3. Assume that $f$ and $g$ are sufficiently general.
  Then the Fano variety of lines $F=F(Y)$ has transversal $A_2$-singularities along a surface $S\subset F$ which is isomorphic to the smooth complete intersection in $H_0:=\{x_0=0\}\cong\PP^4$ given by equations
      \[ \begin{cases}         &f(x_1,\ldots,x_4)=0\\
                            &g(x_1,\ldots,x_4)+ x_5^3=0\ .\\
                            \end{cases}\]  
 Moreover, the blow-up $\wt{F}\to F$ with center $S$ is a resolution of singularities, and $\wt{F}$ is isomorphic to the Hilbert scheme $S^{[2]}$.
 \end{theorem}
 
 \begin{proof} This is part of \cite[Theorem 4.1]{BHS}. Recall that an automorphism of a cubic fourfold $X$ is said to be symplectic if the induced automorphism on $H^4(X,\mathbb{Z})$ acts trivially on  $H^{3,1}(X)$. We note that the cubic $Y$ has an order 3 non-symplectic automorphism  (defined by multiplying $x_5$ with a cubic root of unity). This induces
 an order 3 non-symplectic automorphism of $F$, and $S\subset F$ is the fixed locus of this automorphism. The generality assumptions guarantee that $S$ is non-singular, and that $F$ does not contain planes passing through the cusp $z=[1:0:\ldots:0]$. In that case, it is proven in loc. cit. that there exists a birational morphism $\varphi\colon S^{[2]}\to F$. To define $\varphi$, first consider a pair of distinct points $x_1,x_2\in S$. The lines $\ell_i$, $i=1,2$ spanned by $x_i$ and the cusp $z$ are contained in the cubic $Y$, and the plane spanned by $\ell_1$ and $\ell_2$ cuts $Y$ along a third line $\ell_3$. One defines 
   \[ \varphi(\{x_1,x_2\}):=\ell_3\ .\] 
 It is then shown in loc. cit. that this extends to a birational morphism $\varphi$, contracting a certain ``trident divisor'' $\Psi\subset S^{[2]}$ to the surface $S\subset F$; the morphism $\varphi$ coincides with the blow-up with center $S$.
  \end{proof}

\subsection{The motive of cyclic cubic fourfolds and their Fano variety of lines}

\begin{theorem}[Laterveer \cite{cubic}]\label{cyclic} Let $Y\subset\PP^5$ be a cyclic cubic fourfold, i.e. a smooth cubic fourfold defined by the equation
  \[ g(x_0,\ldots, x_4)+ x_5^3=0\ ,\]
  where $g$ is a homogeneous polynomial of degree 3. Then $Y$ has finite-dimensional motive, in the sense of Kimura \cite{Kim}.
  \end{theorem}
  
  \begin{proof}
This is Theorem 3.1 of \cite{cubic}. The argument consists in showcasing an embedding of motives 
$$h(Y) \hookrightarrow  h(Z)\otimes h(E)(-1)\oplus \bigoplus_i \mathbb{L}(m_i),$$
where $E$ is an elliptic curve, and $Z$ is the cubic fivefold of equation
\[ g(x_0,\ldots, x_4)+ x_5^3 +x_6^3=0\ .\]
Since $Z$ has finite-dimensional motive (by Remark \ref{examples}), and $E$ is a curve, this implies the finite-dimensionality of the motive of $Y$. Remark that a similar construction in cohomology was studied by van Geemen and Izadi \cite{vGI}.

  \end{proof}

  \begin{theorem}[Laterveer \cite{22}]\label{YFY} Let $Y\subset\PP^{n+1}$ be a smooth cubic hypersurface, and let $F=F(Y)$ be the Fano variety of lines contained in $Y$. There is a relation of Chow motives
     \[ h(F)(-2)\oplus\bigoplus_{i=0}^n h(Y)(-i)\cong h(Y^{[2]})\ \ \ \hbox{in}\ \MM_{\rm rat}\ .\]
   \end{theorem}
   
   \begin{proof} This is \cite[Theorem 5]{22}. A slightly stronger result is obtained in \cite[Theorem 3]{FLV3}.
     \end{proof}
     
 We state a proposition for later use:    
     
\begin{proposition}\label{5foldfindim}
Let $Z$ be a singular cubic fivefold with isolated singularities of type $A_n$, and
let $\tilde{Z}\to Z$ be a resolution of singularities. Then the motive of $\tilde{Z}$ is finite-dimensional.
\end{proposition}

\begin{proof} By the work of Hirschowitz-Iyer \cite[Section 1.7]{HI}, any (arbitrarily singular) cubic fivefold $Z$ has $A^{}_0(Z)=A^{}_1(Z)=\QQ$. (For {\em smooth\/} cubic fivefolds, this is also proven in \cite{Par}, \cite{Lew} and \cite{Ot}.) Thus, there exists a curve $C\subset Z$ (obtained as linear section) such that push-forward induces surjections
  \[ A_i(C)\ \twoheadrightarrow\ A_i(Z)\ \ \ (i=0,1)\ .\]
 Denoting $U\subset Z$ the non-singular locus of the cubic, it follows that there are also surjections
 \[ A_i(C\cap U)\ \twoheadrightarrow\ A_i(U)\ \ \ (i=0,1)\ .\]
 The exceptional divisor of the blow-up $\wt{Z}\to Z$ along the singular points consists of a tree of projective spaces. The localization sequence (for the inclusion of the exceptional divisor inside $\wt{Z}$) then implies that there exist curves $C_1,\ldots,C_r\subset \wt{Z}$ with the property that push-forward induces surjections
   \[ A_i(\bigcup_{j=1}^r C_j)\ \twoheadrightarrow\ A_i(\wt{Z})\ \ \ (i=0,1)\ .\]
   In the terminology of \cite{small}, this means that
   \[ \hbox{Niveau}\bigl( A_i(\wt{Z})\bigr) \le 0\ \ \ (i=0,1)\ .\]
   But then the Bloch--Srinivas ``decomposition of the diagonal'' argument (in the precise form of \cite[Thm. 1.7]{small}) implies that
    \[ \hbox{Niveau}\bigl( A_i(\wt{Z})\bigr) \le 1\ \ \ \forall i\ ,\]
    which is equivalent (cf. \cite[Lemma 1.5]{small}) to
    \[ A^i_{AJ}(\wt{Z})=0\ \ \ \forall i\ .\]
    This guarantees that the motive of $\tilde{Z}$ is finite-dimensional, thanks to Vial's result  \cite[Theorem 4]{43}.   
\end{proof}

\subsection{The motive of a Hilbert scheme}


\begin{theorem}[Shen--Vial \cite{SV}]\label{mck} 
		Let $S$ be a K3 surface, and
		let $X$ be the Hilbert square
		$S^{[2]}$.
		Then $X$ admits a self-dual MCK decomposition such that the  induced bigraded
		ring structure $A^\ast_{(\ast)}(X)$ on $A^\ast(X)$ coincides with the
		bigrading
		on $A^\ast(X)$ defined by the ``Fourier transform'' of \cite{SV}, and enjoys
		the
		following properties\,:
		\begin{enumerate}[(i)]
			\item \label{chern} $c_j(X)\in A^{j}_{(0)}(X)$ for all $j$\,;
			\item \label{hard} The multiplication map
			$\cdot D^2\colon A^2_{(2)}(X) \to A^4_{(2)}(X)$
			is an isomorphism for any choice of divisor $D\in A^1(X)$ with
			$\deg(D^4)\neq
			0$\,;
			\item \label{prod}
			The intersection product map
			$A^2_{(2)}(X)\otimes A^2_{(2)}(X) \to A^4_{(4)}(X)$
			is surjective.
		\end{enumerate}	
		Moreover, the incidence correspondence $\Gamma\subset S\times X$ induces an isomorphism
		\[\Gamma_\ast\colon\ \ A^2_{hom}(S)\ \xrightarrow{\cong}\ A^2_{(2)}(X)\ .\]
			\end{theorem}
	
	\begin{proof}		
		We consider the MCK on $X$ constructed in \cite[Theorem 13.4]{SV}\,; its
		relation with the Fourier transform is \cite[Theorem 15.8]{SV}. Statement~\eqref{chern} is \cite[Lemma~13.7(iv)]{SV}, while statement~\eqref{prod} is
		\cite[Proposition~15.6]{SV}.
		
		The ``moreover'' part is implicit in \cite{SV}; let us make it explicit. As in loc. cit., for any $x\in S$ we will write $S_x\subset X$ for the locus of $0$-dimensional subschemes with support containing the point $x$. Let $D\in A^1(X)$ be any divisor with $\deg(D^4)\not=0$. Let $x,y\in S$ be any 2 points. Combining \cite[Proof of Proposition 12.8]{SV} and \cite[Proposition 12.6]{SV}, we find an equality of 0-cycles
		\[  D^2\cdot (S_x - S_y) = q(D)\Bigl( [x,o] - [y,o]\Bigr)\ \ \ \hbox{in}\ A^4(X)\ ,\]
		where $q()$ refers to the Beauville--Bogomolov quadratic form on $X$, and $o\in S$ is any point lying on a rational curve.
		This implies that the composition
		\[  A^2_{hom}(S)\ \xrightarrow{\Gamma_\ast}\ A^2_{(2)}(X)\ \xrightarrow{\cdot D^2}\ A^4_{(2)}(X)\ \xrightarrow{\Gamma^\ast}\ A^2_{hom}(S) \]
		is equal to multiplication with $q(D)\not=0$; in particular the first arrow is an injection.
		As for the surjectivity of $\Gamma_\ast$, this follows from the fact that $A^2_{(2)}(X)$ is generated by expressions of the form $S_x-S_y$ \cite[Theorem 2 and Proposition 15.6]{SV}.
		    \end{proof}


\subsection{Spread argument} A key ingredient in this paper is the machinery of “spread”, as developed by Voisin, and its consequences. This allows us to deal efficiently with algebraic cycles in a family of varieties.

\begin{lemma}[Voisin \cite{Vo}]\label{spread}
Let $\pi:\mathcal{Y}\to B$ be a flat morphism of algebraic varieties, where $B$ is smooth of dimension $r$, and let $Z\in A_n (\mathcal{Y})$ be a cycle. Then the set $B_Z$ of points $t\in B$ such that $Z_t:=Z_{|\YY_t}$ vanishes in $A_{n-r}(\YY_t)$ is a countable union of closed algebraic subsets of $B$.
\end{lemma}

\begin{proof} This is \cite[Lemma 3.1]{Vo}; a proof can be found in  \cite[Proposition 2.4]{Vo4ds}.
\end{proof}

\begin{corollary}\label{spreadkim}
Let $\pi:\mathcal{Y}\to B$ be a smooth projective morphism whose fibers are surfaces with 0 irregularity, such that the very general fiber has finite-dimensional Chow motive. Then the Chow motive of any fiber is finite-dimensional.
\end{corollary}

\begin{proof} 
Let us denote $B^\circ \subset B$ the dense subset where the fibers have finite-dimensional Chow motive, as by hypothesis. Since the fibers are surfaces with $q=0$, finite-dimensionality of their motive is equivalent to the even finite-dimensionality.  This means that, for  $b\in B^\circ$, there exists a non-negative integer $n\in \mathbb{Z}$ such that the projector $\gamma_b^n: h(\mathcal{Y}_b)^{\otimes n} \to h(\mathcal{Y}_b)^{\otimes n}$ defining $\wedge^nh(\mathcal{Y}_b)$ is rationally equivalent to 0. Actually, in our case $n=2+b_2(\mathcal{Y}_b)+1$, since for this value of $n$ we have the vanishing in cohomology. Since $\gamma_b$ is a projector (up to rational equivalence) and the motive $h(\mathcal{Y}_b)$ is assumed to be finite-dimensional, the vanishing follows for the Chow ring. Hence, the cycles $\gamma_b^n$ can be  considered as a family $\Gamma^n \to B^\circ$, whose restriction $\Gamma^n_{|\mathcal{Y}_b}$ to the fiber of $\mathcal{Y}$ over $b$ is the projector $\gamma^n_b$. The closure of $\Gamma^n$ inside $\mathcal{Y}$ gives a cycle over $B$ whose restriction to the very general fiber is rationally equivalent to zero. Then, by Lemma \ref{spread}, the restriction to any fiber is rationally equivalent to zero and the claim is proved.
\end{proof}

\subsection{Chow cohomology}

While dealing with (blow-ups of) singular varieties, we will use Fulton's Chow cohomology \cite[Chapter 17]{F}, because of its properties of fonctoriality. For any (possibly singular) variety $X$, we will write $A_\ast(X)$ for the usual Chow groups (with $\QQ$-coefficients) and $A^\ast(X)$ for the operational Chow cohomology groups (with $\QQ$-coefficients). By construction, $A^\ast(X)$ acts on $A_\ast(X)$; this action is denoted as a cap product. For $X$ non-singular of dimension $n$, there is an isomorphism 
  \[ \cap [X]\colon  A^i(X)\ \xrightarrow{\cong}\ A_{n-i}(X)\ ,\]
  and so for non-singular $X$ we tacitly identify $A^i(X)$ with $A_{n-i}(X)$.

\smallskip

Recall moreover that, thanks to the work of Bloch--Gillet--Soulé \cite{BGS} and Totaro \cite{Tot}, there exists a natural cycle class map

$$A^\ast(X)\to \Gr^W_{\ast} H^\ast(X)$$
(where $W_\ast$ denotes Deligne's weight filtration),
that allows us to define homologically trivial cycles $A^\ast_{hom}(X)$ also in the singular context. Below, we will need the following:

\begin{proposition}\label{chowchow}
Let $X$ be a singular 4-dimensional projective variety, with transversal $A_n$-singularities along a smooth surface $S$ with irregularity $q(S)=0$. Assume that, via the blow-up $p:\tilde{X}\to X$ along $S$, we obtain a smooth fourfold $\tilde{X}$. Then the map
 \[ p^*: A^2_{hom}(X) \ \to \ A^2_{hom}(\widetilde{X}) \]
 is an isomorphism.
 
 In addition, if $H^5(X)=0$ then the maps 

\[ \begin{split}
p_*: A_2^{hom}(\widetilde{X}) & \to  A_2^{hom}(X);\\
\cap [X]: A^2_{hom}(X) & \to  A_2^{hom}(X)\\
\end{split}\]
are injections.
\end{proposition}

\begin{proof} Let $E\subset\wt{X}$ denote the exceptional divisor of the blow-up morphism $p$. Kimura \cite[Theorem 2.3]{Kim0} has shown that there is an exact sequence in operational Chow cohomology
  \[ 0\ \to\ A^2(X)\ \to\ A^2(\wt{X})\oplus A^2(S)\ \to\ A^2(E)\ ,\]
  where all arrows are given by pullbacks. In particular, $p^\ast\colon A^2(X)\to A^2(\wt{X})$ is injective. To show that the restriction of $p^\ast$ to $A^2_{hom}(X)$ is surjective, we consider the commutative diagram with exact rows
  \[ \begin{array}[c]{   cccccccc   }   0 &\to&A^2(X)&\to&A^2(\wt{X})\oplus A^2(S) &\to&  \ima &\to 0\\
                                             &&  && && &\\
                                             &&\downarrow &&\downarrow && \downarrow&\\
                                             && && && &\\
                                             H^3(E)= 0 &\to& \Gr^W_4 H^4(X) &\to& H^4(\wt{X})\oplus H^4(S) &\to& \Gr^W_4 H^4(E) &\to \\
                                             \end{array}\]
       where $\ima$ is shorthand for $ \ima\bigl( A^2(\wt{X})\oplus A^2(S)\to A^2(E)\bigr)$, and $H^3(E)=0$ because of the hypotheses.
       Applying the snake lemma to this diagram furnishes us with an exact sequence
       \begin{equation}\label{snake}   0\ \to\ A^2_{hom}(X)\ \to\ A^2_{hom}(\wt{X})\oplus A^2_{hom}(S)\ \to\ \ima\cap A^2_{hom}(E)\ .\end{equation}
      We now look at the pullback
       \[ A^2_{hom}(S)\ \to\ A^2_{hom}(E)\ ,\]
       and we claim this is a surjection. This claim suffices to show the first statement of the proposition, in view of the exact sequence \eqref{snake}. As for the claim, let $E_1,\ldots,E_n$ denote the irreducible components of $E$. Each $E_j$ is a $\PP^1$-bundle over the regular surface $S$, and hence
       
\begin{eqnarray*}
       A^2(E_j)&= &A^2(S)\oplus A^1(S),\\
       A^2_{hom}(E_j)&=& A^2_{hom}(S).
\end{eqnarray*}
                                 
  The claim then follows from the exact sequence
  \[ 0\ \to\ A^2_{hom}(E)\ \to\ \bigoplus_{j=1}^r A^2_{hom}(E_j)\ \to\ \bigoplus_{j<k} A^2_{hom}(E_j\cap E_k)\ \]       
  (which is again obtained via the snake lemma),
  plus the fact that each intersection $E_j\cap E_k$ is isomorphic to $S$.                        
       
    As for the ``in addition'' part of the proposition, this is proven in similar fashion. We consider the commutative diagram with exact rows
            \[ \begin{array}[c]{   cccccccc   }    &\to&A_2(E)&\to&A_2(\wt{X})\oplus A_2(S) &\to&  A_2(X) &\to 0\\
                                             &&  && && &\\
                                             &&\downarrow  &&\downarrow && \downarrow&\\
                                             && && && &\\
                                             H^5(X)= 0 &\to&  H_4(E) &\to& H_4(\wt{X})\oplus H_4(S) &\to&  H_4(X) &\to \\
                                             \end{array}\]    
                                             where horizontal arrows are pushforwards (the exactness of the upper row follows from \cite[Theorem 1.8]{Kim0}). The left vertical arrow is an injection thanks to Lemma \ref{leftarrow} below.
                            Applying the snake lemma to this diagram (and observing that $A_2^{hom}(S)=0$ for obvious reasons), we obtain an injection
                            \[    p_\ast\colon\ \ A_2^{hom}(\wt{X})\ \hookrightarrow \ A_2^{hom}(X)\ .\]
                            
                  Finally, the projection formula for Chow cohomology \cite[Chapter 17]{F} gives an equality    
                  \[  \cap [X]= p_\ast p^\ast\colon\ \ A^2_{hom}(X)\ \to\ A_2^{hom}(X)    \ ,\]
                  and so the results established above imply that $\cap [X]\colon    A^2_{hom}(X) \to A_2^{hom}(X) $ is also injective.
                  
                  \begin{lemma}\label{leftarrow} In the above set-up, the cycle class map induces an injection
                  \[ A_2(E)\ \hookrightarrow\ H_4(E)\ .\]
                   \end{lemma}    
                   
                   To prove the lemma, we consider the commutative diagram with exact rows
                        \[ \begin{array}[c]{   ccccccccc   }    &\to&\bigoplus_{j<k} A_2(E_j\cap E_k)&\to&  \bigoplus_{j=1}^n A_2(E_j) &\to&  A_2(E) &\to &0\\
                                             &&  && && &&\\
                                             &&\downarrow {\cong} &&\downarrow && \downarrow&&\\
                                             && && && &&\\
                                              &\to&  \bigoplus_{j<k} H_4(E_i\cap E_j) &\to& \bigoplus_{j=1}^n H_4(E_j) &\to&  H_4(E) &\to & \\
                                             \end{array}\]    
                                             (the left vertical arrow is an isomorphism, since $E_i\cap E_j$ is a surface).
                                             Since each $E_j$ is a $\PP^1$-bundle over $S$, we have
                                             \[ A_2(E_j)=A^1(E_j)= A^1(S)\oplus A^0(S) \ \hookrightarrow\ H^2(S)\oplus H^0(S)=H^2(E_j)\ .  \]
                                             An easy diagram chase then concludes the proof of the lemma.
                                             \end{proof}

\section{Main result}

\begin{theorem}\label{main}
Let $S\subset\PP^4$ be a K3 surface defined as a smooth complete intersection with equations
  \[ \begin{cases}  &f(x_0,\ldots,x_3)=0\\
                            &g(x_0,\ldots,x_3)+ x_4^3=0\ ,\\
                            \end{cases}\]
                       where $f$ and $g$ are homogeneous polynomials of degree 2 resp. 3.
                       
                 Then $S$ has finite-dimensional motive, in the sense of Kimura \cite{Kim}.
\end{theorem}

\begin{proof} In view of Corollary \ref{spreadkim}, we may assume the polynomials $f$ and $g$ are sufficiently general. Then, Theorem \ref{bhs} implies that the surface $S$ is related to a cuspidal cyclic cubic fourfold $Y$ via the Fano variety of lines $F=F(Y)$. The key fact that we need to prove the Theorem is the content of the following proposition.

\begin{proposition}\label{corrida}
Let $Y$ be a general cyclic cuspidal cubic fourfold, and $F,\wt{F}$ as in Theorem \ref{bhs}.
There exist correspondences $\alpha\in A^*(\wt{F}\times \wt{G}),\ \beta \in A^*(\wt{G}\times \wt{F})$, where $\wt{G}$ is a finite union of smooth projective varieties with finite dimensional motive, such that the composition

\begin{equation}
A^2_{hom}(\wt{F}) \xrightarrow{\alpha_*} A^*(\wt{G}) \xrightarrow{\beta_*} A^*(\wt{F})
\end{equation}

is the identity.
\end{proposition}

Before we prove Proposition \ref{corrida}, let us show how it implies the Theorem. By Theorem \ref{bhs}, the blow-up $\wt{F}$ is isomorphic to the Hilbert square $S^{[2]}$. By the work of Shen--Vial (cf. Theorem \ref{mck}), we have injections 
  \[ A^2_{hom}(S)\ \hookrightarrow\  A^2_{(2)}(S^{[2]})\ \hookrightarrow\ A^2_{hom}(S^{[2]})= A^2_{hom}(\wt{F})\ .\] 
 Combining this with Proposition \ref{corrida}, we observe that the composition

\begin{equation}
A^2_{hom}(S) \xrightarrow{\gamma_*} A^*(\wt{G}) \xrightarrow{\delta_*} A^2(S)
\end{equation}
is the identity, for some correspondences $\gamma$ and $\delta $. Applying Lemma 3.1 of \cite{SV} to $\Delta_S - \delta \circ \gamma$, this gives a correspondence that induces an injection of motives

$$(\gamma,h): h_{tr}(S) \hookrightarrow h(\wt{G})\oplus \bigoplus_i \mathbb{L}(n_i) \ \ \hbox{in}\ \MM_{\rm rat}\ ,$$
where the last direct summand is just a finite sum of twisted Lefschetz motives. It follows that $S$ has finite-dimensional motive, as requested.
\end{proof}

It remains to prove the proposition:

\begin{proof}(of Proposition \ref{corrida} )

We can find a family $\YY\to B$ of cyclic cubic fourfolds with $B$ smooth one-dimensional such that $Y=Y_0$ is the fiber over $0\in B$, and the other fibers $Y_b$, $b\not=0$ are smooth cyclic cubic fourfolds. Also, $\YY$ is a smooth quasi-projective variety (this follows from smooth base-change, because the universal cubic fourfold $\YY\to \PP H^0(\PP^5,{\mathcal O}_{\PP^5}(3))$ is smooth). Let us write
$\FF\to B$ for the corresponding family of Fano varieties of lines, $\ZZ\to B$ for the corresponding cubic fivefolds (cf. the proof of Theorem \ref{cyclic}). Theorems \ref{cyclic} and \ref{YFY} imply that the fibers $F_b$, $b\not=0$ are Kimura finite-dimensional. More precisely, looking at the argument of Theorems \ref{cyclic} and \ref{YFY}, for any $b\not=0$ one finds inclusions of Chow motives
 
  \begin{equation}\label{inc}   h(F_b)\ \hookrightarrow\ h(Y_b^{2})(2)\oplus \bigoplus h(Y_b)(\ast)\ \hookrightarrow\ h((Z_b)^2\times E^2)(\ast)\oplus \bigoplus h(Z_b\times E)(\ast)\ \ \ \hbox{in}\ \MM_{\rm rat}\ ,\end{equation}
where $E$ is an elliptic curve (which is independent of $Y_b$) and $Z_b$ is the cubic fivefold associated to $Y_b$ (cf. Theorem \ref{cyclic}).   What's more, inspection of the proof of Theorems \ref{cyclic} and \ref{YFY} reveals that the correspondence $\Gamma_b$
  defining the inclusion \eqref{inc} and its left-inverse $\Psi_b$ are {\em generically defined\/} with respect to the base $B^\circ:= B\setminus\{0\}$. That is, there exist relative correspondences 
  \[  \begin{split} &\Gamma^\circ\ \in\ A^\ast\bigl(\FF\times_{B^\circ} (\CC\times_{B^\circ} \CC \times E^2)\bigr)\oplus A^\ast\bigl(\FF\times_{B^\circ} (\CC \times E)\bigr)\ ,\\
                        &\Psi^\circ\ \in\ A^\ast\bigl( (\CC\times_{B^\circ} \CC\times E^2)\times_{B^\circ} \FF\bigr)\oplus A^\ast\bigl( (\CC \times E)\times_{B^\circ} \FF\bigr)\ ,\\
                        \end{split}\]
         with the property that 
    \begin{equation}\label{overcirc}    (\Psi^\circ\circ\Gamma^\circ)\vert_b  =\Delta_{\FF}\vert_b\ \ \ \hbox{in}\ A^4(F_b\times F_b)\ \ \ \forall b\in B^\circ\ .\end{equation}

For the sake of readability, we will set 
  \[ \GG:=(\ZZ \times E)^{2/B} \coprod (\ZZ\times E)^{\coprod r}\ .\]                
A straightforward computation shows that the central cubic fivefold fiber $Z_0$ has one isolated $A_3$ singularity. 
We will need to consider a new family $\wt{\ZZ}\to B$, which is obtained from $\ZZ$ by blowing up the singular point of the central fiber. Remark that the total space of ${\ZZ}$ is smooth, since the universal cubic fivefold is a projective bundle over $\PP^6$, and we just apply a smooth base change. This induces a family $\wt{\GG}\to B$, with the property that the fibers of $\GG$ and $\wt{\GG}$ over $B^\circ$ are the same, but the fiber over $0$ is
\begin{equation}\label{Gdesing}
 \wt{\GG}\vert_0 =\wt{\GG}_0 \cup P,
\end{equation}
  where $\wt{\GG}_0$ is a resolution of singularities of the singular central fiber $\GG_0$, and $P$ is some exceptional divisor which is induced by the blow-up $\wt{\ZZ}\to \ZZ$.

On the other hand, we define the family $\wt{\FF}\to B$ as the blow-up of $\FF$ along the fixed locus of the fiberwise automorphism. By Theorem \ref{bhs}, we have that the central fiber $\wt{\FF}_0$ is smooth and isomorphic to the Hilbert square $S^{[2]}$ of the K3 surface $S$.

\smallskip

For any extension $\Sigma \in A^*(\FF\times_B{\FF})$ such that $\Sigma_{|B^\circ}= \Psi^\circ \circ \Gamma^\circ$, we have equality 
  \[ \Sigma\vert_b=\Delta_\FF\vert_b=\Delta_{F_b}\ \ \ \hbox{in}\ A^4(F_b\times F_b)\ \ \ \ \forall b\in B^\circ \] 
  (this is equality \eqref{overcirc}).
Applying the spread argument (Lemma \ref{spread}), this implies that
the restriction to the central fiber satisfies an equality

\begin{equation}\label{bla}
\Sigma\vert_{0}=\Delta_{F_0}\ \ \hbox{in}\ A_4(F_0\times F_0)\ .
\end{equation}

We will now show how to choose a convenient $\Sigma$ to which we will apply equality \eqref{bla}. Let us denote by $p:\wt{\FF}\to{F}$ the blow-down morphism, $p_0:\wt{F_0}\to F_0$ its restriction to the central fiber, and let us set 
    \[ \widetilde{\Gamma}^\circ:= (p,\ide)^*(\Gamma^\circ)\ \in\ A^4(\wt{\FF}\times_{B^\circ}\wt{\GG})\ ,\  \ \wt{\Psi}^\circ:= (p,\ide)^*(\Psi^\circ)\ \in\ A^4(\wt{\GG}\times_{B^\circ}\wt{\FF})\ .\] We can choose (non-canonical) extensions 
     \[ \widetilde{\Gamma}\ \in\ A^4(\wt{\FF}\times_{B}\wt{\GG})\ ,\  \ \wt{\Psi}\ \in\ A^4(\wt{\GG}\times_{B}\wt{\FF})\ \]    
     with the property that
  $\wt{\Psi}_{|B^\circ}= \wt{\Psi}^\circ$ and $\wt{\Gamma}_{|B^\circ}= \wt{\Gamma}^\circ$. By general properties of the formalism of relative correspondences \cite[Section 8.1.2]{MNP} we have the equality

\begin{equation*}
(\wt{\Psi}\circ \wt{\Gamma})\vert_{0}=\wt{\Psi}\vert_{0}\circ \wt{\Gamma}\vert_{0}\ \ \mathrm{in}\ A^4(\wt{F}_0\times \wt{F}_0). 
\end{equation*}
                        
Now we are in position to define

$$\Sigma:= (p\times p)_*(\wt{\Psi}\circ \wt{\Gamma}).$$

By applying equality \eqref{bla} in $A_4(F_0\times F_0)$ to this choice of $\Sigma$, we obtain a decomposition

\begin{equation}\label{withR}
\Delta_{\wt{F_0}}= \wt{\Psi}_{\vert 0} \circ \wt{\Gamma}_{\vert 0} + R\ \mathrm{in}\ A^4(\wt{F_0}\times \wt{F_0}),
\end{equation} 
where $R$ is a cycle such that $(p_0\times p_0)_*R=0$. This holds since $(p_0\times p_0)_*\Delta_{\wt{F}_0}=\Delta_{F_0}$ and so by functoriality 
  \[(p_0\times p_0)_*\bigl((\wt{\Psi} \circ \wt{\Gamma})_{\vert 0}\bigr)= (p_0\times p_0)_*(\wt{\Psi} \circ \wt{\Gamma})_{\vert 0}\ \ \ \hbox{in}\   A_4(F_0\times F_0)   .\] 
  
 We now claim
 that the correspondence $R$ is such that
   \[ R_\ast=0\colon\ \ A^2_{hom}(\wt{F}_0)\ \to\ A^2_{hom}(\wt{F}_0)\ .\] 
   This claim, together with equality \eqref{withR} shows that
      $$  A^2_{hom}(\wt{F}_0)\ \stackrel{(\wt{\psi}_{|0})_*}{\hookrightarrow} \ A^*(\wt{\GG}_{|0})\ \stackrel{(\wt{\Gamma}_{|0})_*}{\rightarrow}\ A^*(\wt{F}_0)$$ is the identity map. Looking at Equation \eqref{Gdesing} combined with Proposition \ref{5foldfindim}, one sees that $\wt{\GG}_{|0}$ is a finite union of smooth projective varieties with finite-dimensional motive, as requested.
      
 It remains to prove the claim; for this we rely on Proposition \ref{chowchow} for $p_0^*$ and $p_{0*}$  (note that the hypotheses of Proposition \ref{chowchow} are satisfied, as $F_0$ is the Fano variety of lines on a singular cubic fourfold, and so $F_0$ has no odd-degree cohomology \cite[Theorem 6.1]{GS}). This proposition reduces the claim to proving that the composition
   \[  A^2_{hom}(F_0)      \ \xrightarrow{(p_0)^\ast}\  A^2_{hom}(\wt{F}_0)\ \xrightarrow{R_\ast}\  A^2_{hom}(\wt{F}_0)\ \xrightarrow{(p_0)_\ast}\ A_2^{hom}(F_0) \]
   is the zero map. But the projection formula for Chow cohomology \cite[Chapter 17]{F} implies that this composition is the same as the composition
   
    \[   A^2_{hom}(F_0)\ \xrightarrow{(\pi_1)^\ast}\ A^2_{hom}(F_0\times F_0)\ \xrightarrow{\cap (p_0\times p_0)_\ast (R)}\ A_2^{hom}(F_0\times F_0)\ \xrightarrow{(\pi_2)_\ast}\ A_2^{hom}(F_0)\ , \]
where $\pi_1,\pi_2\colon F_0\times F_0\to F_0$ denote the projections on the 2 factors.
 This last composition is the zero map, as $(p_0\times p_0)_*R=0$ by construction. The claim is proven, and hence so is Proposition \ref{corrida}.
%
  \end{proof}

\begin{remark} As the referee points out, for any $0\not=b\in B$, the fibre of the family $\wt{\FF}\to B$ is the blow-up $\wt{F}_b$ of $F_b$ at the fixed locus of the order 3 automorphism.
This fixed locus is the Fano surface of lines of the smooth cubic threefold associated to $F_b$, and so (using Remark \ref{examples}) the fibre $\wt{F}_b$ has finite-dimensional motive. This gives an alternative way of proving Proposition \ref{corrida}.
\end{remark}

\section{The Kuga--Satake Hodge conjecture}

In this section, we prove the Kuga--Satake Hodge conjecture for our K3 surfaces. The argument is very similar to that of Theorem \ref{main}: we use that the K3 surface is related to cyclic cubic fourfolds, plus the fact that the Kuga--Satake Hodge conjecture is known for cyclic cubic fourfolds (thanks to work of van Geemen--Izadi).

\begin{theorem}[Kuga--Satake] Let $V$ be a polarized Hodge structure of K3 type. There exists an abelian variety $KS(V)$, the {\em Kuga--Satake variety\/} associated to $V$, such that there is an
embedding of Hodge structures
  \[ \iota_V\ \colon\ V\ \hookrightarrow\ H^2(KS(V)\times KS(V),\QQ)\ .\]
\end{theorem}

\begin{proof} This is explained in \cite[Chapter 4]{Huy} and \cite{vG}. The original reference is \cite{KS}.
\end{proof}

The following is a special case of the Hodge conjecture:

\begin{conjecture}(``Kuga--Satake Hodge conjecture'')\label{kshc} Let $V=H^{2r}(X,\QQ)$ be a Hodge structure of K3 type, for some smooth projective variety $X$. The embedding $\iota_V$ is induced by a correspondence in $A^\ast(X\times KS(V)\times KS(V))$.
\end{conjecture} 

\begin{remark} Conjecture \ref{kshc} is widely open for K3 surfaces (i.e. when $V=H^2(S,\QQ)$ for $S$ a K3 surface). It is known for $S$ a (resolution of a) double plane branched along 6 lines \cite{Par0}, for $S$ a (resolution of a) complete intersection in $\PP^4$ with 15 double points \cite{ILP}, for $S$ a cyclic quartic in $\PP^3$ \cite{vG}, and for some other K3 surfaces with Picard number 16 \cite{Flo}.
\end{remark}

\begin{theorem}\label{main2} Let $S\subset\PP^4$ be a K3 surface defined as a smooth complete intersection with equations
  \[ \begin{cases}  &f(x_0,\ldots,x_3)=0\\
                            &g(x_0,\ldots,x_3)+ x_4^3=0\ ,\\
                            \end{cases}\]
                       where $f$ and $g$ are homogeneous polynomials of degree 2 resp. 3.
                       
      Then Conjecture \ref{kshc} is true for $V=H^2(S,\QQ)$.                 
\end{theorem}

\begin{proof} The Kuga--Satake construction can be performed in a family. Precisely, let $V\to B$ be a polarized variation of Hodge structure of K3 type. Then it is known \cite[Proposition 6.4.10]{Huy} that there exists a finite \'etale base-change $B^\prime\to B$ and an abelian scheme $A\to B^\prime$ such that the fiber $A_b$ over $b\in B^\prime$ is the Kuga--Satake variety associated to the Hodge structure $V_b$. Applying this to the family of smooth complete intersections as in Theorem \ref{main2}, we see that it suffices to prove that for a {\em generic\/} K3 surface $S$ as in Theorem \ref{main2}, there exists a correspondence that is {\em generically defined\/} (with respect to the base change $B^\prime$) inducing the Kuga--Satake embedding.

So let us now assume that $S=S_0$ is a generic member of the family of complete intersections as in Theorem \ref{main2}. Then Theorem \ref{bhs} applies, and gives us a 1-dimensional family $\YY\to B$ of cyclic cubic fourfolds, such that $Y_b$ is smooth for $b\not=0$ and $Y_0$ is such that the Fano variety of lines $F(Y_0)$ has a resolution of singularities isomorphic to the Hilbert scheme $S_0^{[2]}$. As above, let us write $\FF\to B$ for the family of Fano varieties of lines, and 
  \[  \pi\ \colon\ \wt{\FF}\to B\] for the blow-up of $\FF$ with center the fixed locus of the order 3 fiberwise automorphism (so that $\FF_0$ is isomorphic to $S_0^{[2]}$, cf. Theorem \ref{bhs}). 
  
  The local system $R^2\pi_\ast \QQ\to B$ is a variation of Hodge structure of K3 type, and so the above-mentioned ``Kuga--Satake in family'' result \cite[Proposition 6.4.10]{Huy} applies.
  That is, there exists a finite base-change $B^\prime \to B$ and a family $A\to B^\prime$ of associated Kuga--Satake varieties. Let $B^\prime_\circ\subset B^\prime$ denote the pre-image of $B^\circ=B\setminus \{0\}$. Then for any $b\in B^\prime_\circ$, the fiber $\wt{F}_b$ of $\wt{\FF}\to B^\prime$ over $b$ is a blow-up of the Fano variety of lines of a smooth cyclic cubic fourfold. A lovely result of van Geemen--Izadi \cite[Corollary 5.3]{vGI} settles the Kuga--Satake Hodge conjecture for the smooth cyclic cubic fourfold $Y_b$. Moreover, inspection of their proof (cf. \cite[Theorem 2.8]{cubic}) reveals that the correspondence they produce is actually generically defined (with respect to $B$). In view of the isomorphisms
   \[ H^4_{tr}(Y_b,\QQ)_{}\cong H^2_{tr}(F_b,\QQ)_{}\cong H^2_{tr}(\wt{F_b},\QQ)_{} \]
(where the first isomorphism is given by the Abel--Jacobi map, and the second follows from the blow-up formula), this gives a generically defined Kuga--Satake correspondence for $\wt{F_b}$, $b\in B^\prime_\circ$. Extending this relative correspondence to $B^\prime$, we obtain the Kuga--Satake Hodge conjecture for $\wt{F_0}\cong S_0^{[2]}$. Because of the natural isomorphism (induced by a generically defined correspondence)
     \[ H^2_{tr}(S_0,\QQ)\cong H^2_{tr}(S_0^{[2]},\QQ)\ ,\] 
     this also gives a (generically defined) Kuga--Satake correspondence for $S_0$. This closes the proof.
\end{proof}

\begin{remark}\label{KSChow} A general Hilbert schemes argument due to Voisin implies that the inverse to the Kuga--Satake correspondence is also induced by a generically defined correspondence (cf. \cite[2.3]{cubic}). The above argument then gives the following: for any K3 surface $S$ as in Theorem \ref{main}, with Kuga--Satake variety $A$, there exist correspondences $\Gamma,\Psi$ such that the composition
  \[ H^2(S,\QQ)\ \xrightarrow{\Gamma_\ast}\ H^2(A\times A,\QQ)\ \xrightarrow{\Psi_\ast}\ H^2(S,\QQ) \]
  is the identity. Since we know that $S$ is Kimura finite-dimensional (Theorem \ref{main}), it follows that there is an embedding of Chow motives
  \[  \Gamma\ \colon\ h^2_{tr}(S)\ \hookrightarrow\ h^2(A^2)\ \ \ \hbox{in}\ \MM_{\rm rat}\ \]
  (where $h^\ast(A^2)$ refers to the Deninger--Murre decomposition of the motive \cite{DM}),
  and so in particular there is an embedding of Chow groups
  \[ \Gamma_\ast\colon\ \ A^2_{hom}(S)\ \hookrightarrow\ A^2_{(2)}(A^2)\ \]
  (where $A^\ast_{(\ast)}(A^2)$ refers to the Beauville decomposition of the Chow ring \cite{Beau}).
\end{remark}

\section{Voisin's conjecture}

\begin{conjecture}[Voisin \cite{Vo0}]\label{vo} Let $X$ be a strict Calabi--Yau variety of dimension $n$ (i.e. $h^{i,0}(X)$ is 0 for $0<i<n$ and $h^{n,0}(X)=1$). Then for any 2 zero-cycles $a,b\in A^n_{hom}(X)$ there is equality
  \[ a\times b + (-1)^{n-1} b\times a=0\ \ \ \hbox{in}\ A^{2n}(X\times X)\ .\]
  (Here the notation $a\times b$ is by definition $\pi_1^\ast(a)\cdot \pi_2^\ast(b)$, where $\pi_i\colon X\times X\to X$ are the projections on the two factors.)
  \end{conjecture}
  
  \begin{remark} For background and motivation for Voisin's conjecture, cf. \cite[Section 4.3.5.2]{Vo}. This conjecture is still open for K3 surfaces. For some special cases where Conjecture \ref{vo} is known, cf. \cite{Vo0}, \cite{18}, \cite{19}, \cite{20}, \cite{21}, \cite{38}, \cite{47}.
    \end{remark}
    
  \begin{theorem}\label{main3} Let $S\subset\PP^4$ be a K3 surface defined as a smooth complete intersection with equations
  \[ \begin{cases}  &f(x_0,\ldots,x_3)=0\\
                            &g(x_0,\ldots,x_3)+ x_4^3=0\ ,\\
                            \end{cases}\]
                       where $f$ and $g$ are homogeneous polynomials of degree 2 resp. 3.
                       
      Then Conjecture \ref{vo} is true for $S$.               
  \end{theorem}  
  
  \begin{proof} We apply the following general criterion:
  
  \begin{proposition}[\cite{47}\label{crit}] Let $X$ be a strict Calabi--Yau variety of dimension $n$. Assume that there is an isomorphism of Chow motives
    \[ h(X)\cong \bigoplus_\chi h^n(A_\chi)\oplus \bigoplus \one(\ast)\ \ \hbox{in}\ \MM_{\rm rat}\ ,\]
    where $A_\chi$ are abelian varieties, and $h^\ast(A_\chi)$ refers to the Deninger--Murre decomposition \cite{DM}. Then Conjecture \ref{vo} is true for $X$.
   \end{proposition}
(This is \cite[Lemma 2.1]{47}.) This criterion applies to $S$ in view of the results we established, cf. Remark \ref{KSChow} above.
   \end{proof}

\section{Further consequences, and closing remarks}\label{final}

We start this section by proving the consequence announced in the introduction:

\begin{corollary}\label{cor} Let $S$ be a K3 surface with an order 3 non-symplectic automorphism $\sigma$, and assume that the fixed locus of $\sigma$ contains a curve. Then $S$ has finite-dimensional motive.
\end{corollary}

\begin{proof} Thanks to the classification result of Artebani--Sarti (Theorem \ref{as}), we know that $S$ is in the closure of $\MM_{0,1}$ or in the closure of $\MM_{0,2}$.
We have shown (Theorem \ref{main}) that any surface in $\MM_{0,1}$ is Kimura finite-dimensional; thanks to the spread argument (Corollary \ref{spreadkim}), the same is true for any surface in the closure of $\MM_{0,1}$.

As for $\MM_{0,2}$, thanks to the work of Kondo (Proposition \ref{kon}) we know that the general element $S^\prime\in \MM_{0,2}$ is the Jacobian elliptic fibration of an element $S\in\MM_{0,1}$.
This implies that $S^\prime$ and $S$ are twisted derived equivalent, i.e. there is an equivalence of derived categories
  \[ D^b(S^\prime, \alpha)\cong D^b(S)\ ,\]
  where $\alpha$ is a Brauer class \cite[Remark 4.9]{Huy}. This implies that these surfaces have isomorphic Chow motives:
  \[ h(S^\prime)\cong h(S)\ \ \ \hbox{in}\ \MM_{\rm rat} \]
 \cite[Theorem 1]{FV} (cf. also the argument of \cite{Huy0}), and so $S^\prime$ is Kimura finite-dimensional. Applying once more the spread argument (Corollary \ref{spreadkim}), we find that the same is true
 for any surface in the closure of $\MM_{0,2}$. This ends the proof.
 \end{proof}

Also the results from Section 4 and 5 can be extended to $\mathcal{M}_{0,1}$ and $\mathcal{M}_{0,2}$. For instance, Theorem \ref{main2} extends harmlessly to the component $\mathcal{M}_{0,1}$ of the moduli space of K3 surfaces with an order 3 non-symplectic automorphism.

\begin{proposition}\label{ksh01}
Conjecture \ref{kshc} holds true for $V=H^2(S,\QQ)$, for all K3 surfaces inside the component $\mathcal{M}_{0,1}$.
\end{proposition}

\begin{proof}
As we have observed in the proof of Theorem \ref{main2}, it is enough that the Kuga-Satake correspondence is generically defined for the family $\mathcal{M}_{0,1}$. This is true by the same argument used in the proof of Theorem \ref{main2}, hence the Conjecture holds true for all K3 surfaces in $\mathcal{M}_{0,1}.$ 
\end{proof}

\begin{corollary}\label{rem01}
All the statements of Remark \ref{KSChow} hold true for a K3 surface inside $\mathcal{M}_{0,1}$.
\end{corollary}

\begin{proof}
In order to conclude, it is enough to observe that, by Corollary \ref{cor}, all K3 surfaces in $\mathcal{M}_{0,1}$ have finite-dimensional motive and, by Proposition \ref{ksh01}, their Kuga-Satake correspondence is algebraic. The same argument of Remark \ref{KSChow} then applies.
\end{proof}

\begin{proposition}
Conjecture \ref{kshc} holds for all K3 surfaces inside $\mathcal{M}_{0,2}$.
\end{proposition}
 
\begin{proof}
By \cite{KS}, we know that Kuga-Satake varieties exist for all the K3 surfaces in $\mathcal{M}_{0,2}$. We need to show that the embedding is given by an algebraic correspondence. We will do this by factoring it through $\mathcal{M}_{0,1}$. That is: as above, we may reduce to a general element $S^\prime$ in $\mathcal{M}_{0,2}$. Then, thanks to Kondo's work  there exists a correspondence $\delta \in A^*(S'\times S)$  that induces the isomorphism of motives $h(S')\cong h(S)$, for a certain $S\in \mathcal{M}_{0,1}$. Let us now denote by $\gamma_S\in A^*(S\times KS(V)\times KS(V))$ the Kuga-Satake correspondence for $S\in \mathcal{M}_{0,1}$. In order to conclude, it is enough to recall \cite[Section 4]{Huy} that the Kuga-Satake varieties of K3 surfaces with isomorphic Chow motives are isogenous. Hence the Kuga-Satake correspondence of $S'\in \mathcal{M}_{0,2}$ is the composition of $\delta,\ \gamma_S$ and the correspondence defining the isogeny, hence it is algebraic as well.
\end{proof} 

By Corollary \ref{rem01}, we have the following useful consequence.

\begin{corollary}
Conjecture  \ref{vo} holds true for all K3 surfaces in $\mathcal{M}_{0,1}$.
\end{corollary}

We observe now that Remark \ref{KSChow} holds also for $\mathcal{M}_{0,2}.$ Hence one can argue as in Section 5, and by the results of Remark \ref{KSChow}, combined with Proposition \ref{crit}, we obtain that

\begin{corollary}
Conjecture  \ref{vo} holds true for all K3 surfaces in $\mathcal{M}_{0,2}$.
\end{corollary}

 \begin{remark}\label{pity} We have tried in vain to establish Kimura finite-dimensionality for surfaces in $\MM_{3,0}$. These surfaces can also be described explicitly as certain complete intersections in $\PP^4$ \cite[Proposition 4.7]{AS}. Just as in the argument of Theorem \ref{main}, we can relate them to certain singular cubic fourfolds; these singular cubic fourfolds are degenerations of certain smooth cubic fourfolds with an order 3 non-symplectic automorphism (this is the family described in \cite[Example 6.6]{BCS}). The problem is that (contrary to the situation of Theorem \ref{main}), these smooth cubic fourfolds are {\em not\/} known to be Kimura finite-dimensional, and so our argument breaks down here.
  \end{remark}

   \begin{remark}
From \cite{AS}, we know that the generic surface in the two families we consider has Picard number 2. We observe that 
examples of K3 surfaces with Picard number 2 and finite-dimensional motive have already been discovered in \cite[Corollary 2]{Ped}. Pedrini's examples are isolated in the moduli space, because of \cite[Theorem 5]{LSY}, while ours move in 9-dimensional families.
\end{remark} 
 
  \vskip1.0cm
\begin{nonumberingcon} 
The authors have no competing interests to declare that are relevant to the content of this article.
\end{nonumberingcon}

 \vskip1cm
\begin{nonumberingt} 
Thanks to Samuel Boissi\`ere for an inspiring talk on \cite{BHS} in Strasbourg (May 2023). RL thanks MB for an invitation to sunny Montpellier (June 2023), where this work was initiated. Thanks to the referee for helpful and insightful comments.
\end{nonumberingt}

\end{document}